\documentclass{article}
\usepackage[utf8]{inputenc}
\usepackage{amsmath,amsfonts,amssymb,amscd,amsthm,xspace}
\usepackage[shortlabels]{enumitem}
\usepackage{tikz-cd}
\usepackage{imakeidx}
\usepackage{indentfirst}
\usepackage{phoenician}
\theoremstyle{plain}
\newtheorem{theorem}{Theorem}[section]
\newtheorem{ex}[theorem]{Example}
\newtheorem{coro}[theorem]{Corollary}
\newtheorem{lemma}[theorem]{Lemma}
\newtheorem{prop}[theorem]{Proposition}

\theoremstyle{definition}
\newtheorem{defi}[theorem]{Definition}

\usepackage{mathtools}

\theoremstyle{remark}
\newtheorem{remark}[theorem]{Remark}
\newtheorem{notation}[theorem]{Notation}

\usepackage{verbatim}

\usepackage{hyperref}
\hypersetup{colorlinks,allcolors=black,breaklinks}

\newcommand{\adj}{\textup{adj}}
\usepackage{mathdots}

\usepackage[a4paper,width=150mm,top=25mm,bottom=25mm]{geometry}

\title{Topologically simple infinite matrix groups indexed by ordered sets}
\author{João V. P. e Silva\footnote{Belo Horizonte, Minas Gerais, Brazil - https://orcid.org/my-orcid?orcid=0000-0002-1441-1930 }\\ email: \href{mailto:joaovitorps@outlook.com}{joaovitorps@outlook.com}}
\date{}

\begin{document}

\maketitle

\begin{abstract}
\noindent This article focuses on the study of the group of units of incidence rings, which is a class of infinite matrix groups indexed by ordered sets,  on a topological perspective. We first show when these groups can inherit the topological structure from the incidence rings. It is later shown that infinite matrix groups of topological fields can be used to build simple topological matrix groups, generalizing a result proven in ``Topologically simple, totally disconnected, locally compact infinite matrix groups''. We finish by relating the structure of these groups with elementary totally disconnected, locally compact groups, an important class for the study of totally disconnected, locally compact groups.
\indent The paper contains 7081 words.
\end{abstract}

\section{Introduction}


The study of infinite matrices traces back to Poincaré in the discussion of the well known Hill's equation. In \cite{history_infinite_matrices}, Bernkopf gives an introduction to the history of the study of infinite matrices and their importance for operator theory. \\
\indent At the present paper we focus on groups of units of incidence rings, a class of rings described using locally finite ordered sets, that is,\ the sets $[x,y]=\{z;\ x\leqslant z\leqslant y\}$ are finite for every $x,y$ in the ordered set, first defined in \cite{Matrices1970}. \\ 
\indent A generalization of the construction of matrix groups and rings from \cite{simple_matrices} is given for ordered sets that are not $\mathbb{Z}$-like \cite[Definition 2.1]{simple_matrices}. We also generalize the construction of non-compactly generated, non-discrete, simple groups \cite[Theorem 5.2]{simple_matrices} for a bigger class of groups that are not necessarily totally disconnected, nor locally compact, nor second countable.\\
\indent Groups similar to the ones studied here have been described previously. The groups denoted here as $\textup{aGL}_\mathbf{N}(\mathbb{F}_q)$,\ for $\mathbf{N}$ the set of naturals with usual order and $\mathbb{F}_q$ the finite field with $q$ elements, was first built at \cite{first_almost_upper_triang} as the group $\textup{GLB}$. Some first study on the representation of such groups was done at \cite{representation_glb}.\\
\indent Study of the group $\textup{GLB}$ and similarly defined groups over more general rings are also described in \cite{Hou2017CommutatorSO,GUPTA20124279,GUPTA201585,10.2307/2699674}.\\
\indent We also prove sufficient conditions for the group of units of an incidence ring to be quasi-discrete or solvable. It then follows that if $G$ contains a matrix group with such property as an open, compact subgroup, then $G$ is always an elementary topological group in the Wesolek sense \cite{elementary_first}.\\
\indent All rings in this article are assumed to be associative with identity. Rings are assumed to be non-trivial. Topological rings and groups are always Hausdorff spaces.

\subsection{Structure of the article}

Section \ref{background} of this article will focus on giving the necessary background for the definitions and results from the article. In Subsection \ref{ordered_sets} some definitions for ordered sets are given. These will be central for the definition of the incidence rings, their groups of units,\ and to prove results for such classes of objects. Subsection 2.2 gives some results from category theory that will be used in a later section.\\
\indent Section \ref{groups} focuses on the groups of units of an incidence ring. A description of such groups is given, and a generalization of the construction of $\textup{GLB}$ from \cite{first_almost_upper_triang} and $\textup{AU}_{\Lambda}(\mathbf{F}_q)$ from \cite{simple_matrices} is also given. The main focus of the section is generalizing the construction of simple groups as in \cite{simple_matrices} and showing it can be done for an arbitrary topological field. We also show properties of topological groups that are locally the group of units of incidence rings of ``not too big'' preordered sets.\\


\section{Background}\label{background}

\subsection{Ordered sets}\label{ordered_sets}

\indent For the context of our rings and groups,\ a definition that proves necessary is the following:

\begin{defi}
Let $(\Lambda,\  \preceq)$ be a proset (preordered set). Given $s_1  \preceq s_2 \in \Lambda$ we define $[s_1,\ s_2]:=\{s\in \Lambda:\ s_1  \preceq s  \preceq s_2\}$. We will call these subsets of $\Lambda$ \textbf{intervals}. We say that the proset is \textbf{locally finite} if for every $s_1  \preceq s_2\in S$ the set $[s_1,\ s_2]_\preceq$ is finite.
\end{defi}
\indent This locally finite type of property allows us to define multiplication and addition on these rings in a similar way to the finite-dimensional matrix rings.\\
\indent As the central object of this text is algebraic structures defined in relation to an ordered set,\ defining some subsets and relations in the ordered sets will be essential for simplifying the proofs. This subsection will focus on giving these basic definitions and providing some results about the structure of ordered sets and their subsets.
\begin{notation}
Let $\Lambda$ be a proset and $s_1\preceq s_2\in \Lambda$. We denote:
\begin{itemize}
    \item $s_1\precnsim s_2$ if $s_1 \preceq s_2$ and $s_1\nsim s_2$.
    \item $s_1\precnapprox s_2$ if $s_1 \preceq s_2$ and $s_1\neq s_2$.
    \item $s_1\npreceq s_2$ if it is not the case $s_1\preceq s_2$.
\end{itemize}
\end{notation}

\indent Notice that for $\Lambda$ a proset the set $[s,\ s]$ is the set of all elements equivalent to $s$ under the preorder. Hence it is not always the case $[s,\ s]=\{s\}$. Notice that if for every $s\in \Lambda$,\ $[s,\ s]=\{s\}$ then $\Lambda$ is a poset (partially ordered set).\\

\begin{defi}\label{defi:general poset one} 
Given a proset $\Lambda$,\ we define the following:
\begin{itemize}
    \item Given $s\in \Lambda$,\ we define the \textbf{neighbourhoods} of $s$ as:
    $$\mathcal{N}_{n}(s):=\left\{\begin{array}{ll}
       [s,\ s]  & \mbox{for }n=0 \\
       \{t\in \Lambda:\ t  \preceq s \mbox{ or }s  \preceq t\} & \mbox{for }n=1 \\
    \bigcup_{t\in\mathcal{N}_{n-1}(s)}\mathcal{N}_1(t) & \mbox{for } n>1
    \end{array}\right.$$
    The \textbf{full collection of neighbours} $s$ is defined as $\mathcal{N}_\omega(s)=\bigcup_{n\in\mathbb{N}}\mathcal{N}_n(s)$.
    \item We say $s_1,\ s_2\in \Lambda$ are \textbf{independent} if $s_1\notin \mathcal{N}_{\omega}(s_2)$.
    \item A subset $\Lambda'\subset \Lambda$ is said to be \textbf{closed under intervals} if for every $s_1  \preceq s_2\in 
    \Lambda'$ then $[s_1,\ s_2]\subset \Lambda'$.
    \item Given $\Lambda'\subset \Lambda$ a subset,\ we say $\Lambda'$ is \textbf{convex} if it is closed under intervals and for every $s_1,\ s_2\in \Lambda$ there are $t_1,\ t_2,\ \ldots,\ t_n\in \Lambda'$ such that $t_1\in \mathcal{N}_1(s_1)$,\ $t_{i+1}\in \mathcal{N}_1(t_i)$,\ $s_2\in \mathcal{N}_1(t_n)$ for $1\leqslant i\leqslant n-1$. In other words,\ there is a path of intervals inside $\Lambda'$ connecting $s_1$ and $s_2$.
    \item Given $\{\Lambda_i\}_{i\in I}$ a collection such that for all $i\in I$,\ $\Lambda_i\subset \Lambda$ is convex in $\Lambda$ and for $i\neq j\in I$ then $\Lambda_i\bigcap \Lambda_j=\emptyset$,\ then $\{\Lambda_i\}_{i\in I}$ is called a \textbf{locally convex collection} of $\Lambda$. 
    \item We denote by $\Gamma(\Lambda)$ the \textbf{set containing all finite, locally convex collections of $\Lambda$}.
\end{itemize}
\end{defi}


\begin{defi}
Let $\Lambda$ be a proset and $s\in\Lambda$. We say that:
\begin{itemize}
    \item $s$ is a \textbf{maximal element} if given $t\in\Lambda$ is such that $s  \preceq t$ then $s=t$.
    \item $s$ is a \textbf{minimal element} if given $t\in\Lambda$ is such that $t  \preceq s$ then $s=t$.
    \item $s$ is the \textbf{maximum element} if for all $t\in \Lambda$ we have $t  \preceq s$.
    \item $s$ is the \textbf{minimum element} if for all $t\in \Lambda$ we have $s  \preceq t$.
\end{itemize} 
\end{defi}

The following examples illustrate the definition of intervals,\ neighbourhoods, and convex subsets as denoted above.

\begin{ex}\label{ex:basic posets}\label{ex:omega complexity degree}
\begin{itemize}
\item $\mathbf{Q}$: The rational numbers with usual order is a partially ordered sets that is not locally finite.
\item $\mathbf{Zig}$: This is the poset with elements in the integers and order given by $2k>2k+1$ and $2k> 2k-1$,\ for $k$ an integer.\\
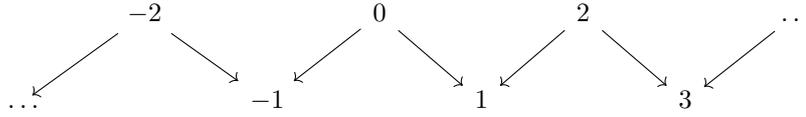
$$\begin{tikzcd}
       & -2 \arrow[ld] \arrow[rd] &    & 0 \arrow[ld] \arrow[rd] &   & 2 \arrow[ld] \arrow[rd] &   & \ldots \arrow[ld] \\
\ldots &                          & -1 &                         & 1 &                         & 3 &                  
\end{tikzcd}$$
The only intervals for $\textbf{Zig}$ are of the form $[n,\ n]=\{n\}$,\ $[2n,\ 2n-1]=\{2n,\ 2n-1\}$ and $[2n,\ 2n+1]=\{2n,\ 2n+1\}$ for some $n\in\mathbb{Z}$. In this case,\ $\mathcal{N}_0(0)=\{0\}$,\ $\mathcal{N}_n(0)=\{i\}_{-n\leqslant i\leqslant n}$ and,\ because no two elements are independent,\ $\mathcal{N}_\omega(0)=\mathbf{Zig}$. A subset $S\subset \mathbf{Zig}$ is convex if,\ and only if,\ there are $n_1,\ n_2\in\mathbf{Z}\bigcup\{\infty,\ -\infty\}$ such that $S=\{i\}_{n_1< i < n_2}$. This poset has no independent element. Every element is either maximal or minimal,\ but it doesn't have any maximum or minimum element.

\item $\mathbf{N}^*_d$: The non-zero natural numbers with order given by $a\leqslant b$ if $b=ak$ for some $k$ natural number.\\
Order on the first $6$ natural numbers.
$$
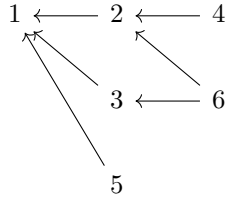
\begin{tikzcd}
1 & 2 \arrow[l]   & 4 \arrow[l]            \\
  & 3 \arrow[lu]  & 6 \arrow[l] \arrow[lu] \\
  & 5 \arrow[luu] &                       
\end{tikzcd}$$
The intervals of this poset can be described using arithmetic properties. For example $[\leqslant 30]=\{1,\ 2,\ 3,\ 5,\ 6,\ 10,\ 15,\ 30\}_{\mathbf{N}_d^*}$ (the divisors of $30$) and $[30\leqslant]_{\mathbf{N}_d^*}=\{30n\}_{n\in\mathbf{N}_d^*}$ (multiples of $30$). The interval $[3,\ 30]=\{3,\ 6,\ 15,\ 30\}$,\ which are all multiples of $3$ that divide $30$. Observe that $\mathcal{N}_1(1)=\mathbf{N}_d^*$ but if $n\neq 1$ then $\mathcal{N}_1(n)$ is the set of all multiples and divisors of $n$,\ but not all $\mathbf{N}^*$. But,\ as $1$ is a divisor of all natural numbers,\ $\mathcal{N}_2(n)=\mathbf{N}_d^*$ for all $n\in \mathbf{N}_d^*$. The element $1$ is the minimum, but there is no maximal element.
\end{itemize} 
\end{ex}

\indent We will also denote some commonly used ordered sets as follows:

\begin{itemize}
\item $\mathbf{n<}$: The poset with elements $\{0,\ 1,\ \ldots,\ n-1\}$ and order given by $0< 1< \ldots< n-1$.

\item $S$: the poset with elements in the set $S$ and order given by equality. This is the poset with elements on $S$ such that all elements are independent.\\
Example: $S=\{0,\ 1,\ 2\}$.
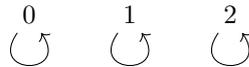
$$\begin{tikzcd}
0 \arrow[loop, distance=2em, in=305, out=235] & 1 \arrow[loop, distance=2em, in=305, out=235] & 2 \arrow[loop, distance=2em, in=305, out=235]
\end{tikzcd}$$

\item $\overline{S}$: the proset with elements in the set $S$ and order given by $i  \preceq j$ for all $i,j$. In other words,\ it is the proset where all elements are equivalent. In the case the set is $\{0,\ 1,\ 2,\ \ldots,\ n-1\}$ we will just denote it as $\textbf{n}$.\\
Example: $\textbf{3}$
$$
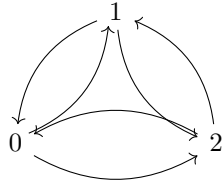
\begin{tikzcd}
                                                             & 1 \arrow[rdd, bend right] \arrow[ldd, bend right] &                                                  \\
                                                             &                                                   &                                                  \\
0 \arrow[rr, bend right] \arrow[ruu, bend right, shift left] &                                                   & 2 \arrow[ll, bend right] \arrow[luu, bend right]
\end{tikzcd}$$

\item $\mathbf{N}$: The natural numbers with the usual order,\ that is,\ $0< 1< 2< \ldots < n < \ldots$.

\item  $\mathbf{Z}$: The integers with the usual order,\ that is,\ $\ldots < -1< 0 < 1 < \ldots.$
\end{itemize}

\begin{defi}\label{defi:disjoint union prosets}
Let $(\Lambda,\  \preceq)$ a proset. If there is a collection $\{\Lambda_i\}_{i\in I}$ of convex subsets of $\Lambda$ such that $\Lambda=\bigcup_{i\in I}\Lambda_i$,\ for every $i\neq j$ we have $\Lambda_i\bigcap \Lambda_j=\emptyset$ and for every $s_i\in \Lambda_i$,\ $s_j\in \Lambda_j$ we have that if $i\neq j$ then $s_i$ and $s_j$ are independent,\ we say $\Lambda$ is the \textbf{internal disjoint union} of $\{\Lambda_i\}_{i\in I}$. We denote it by $\Lambda=\bigsqcup_{i\in I} \Lambda_i$.\\
\indent If $\Lambda$ cannot be written as disjoint union of two non-empty subsets we say $\Lambda$ is \textbf{irreducible}.\index{disjoint union of POSETS|ndx} \index{irreducible POSET|ndx}
\\ \indent Let $\{(\Lambda_i,\  \preceq_i)\}_{i\in I}$ a collection of prosets. We define the \textbf{external disjoint union} of this collection as the proset $(\bigsqcup_{i\in I} \Lambda_i,\  \preceq_{\sqcup})$ with elements in $\bigsqcup_{i\in I}\Lambda_i$ and order $  \preceq_\sqcup$ given by $s_1  \preceq_\sqcup s_2\in \bigsqcup_{i\in I}\Lambda_i$ if,\ and only if,\ there is $i\in I$ such that $s_1  \preceq_i s_2\in \Lambda_i$.
\end{defi}

For the ones familiar with category theory,\ the disjoint union is the coproduct on the category of prosets. Given the posets $\mathbf{2}\! <=\{0,\ 1\}$ and $\mathbf{3}\!<=\{0',\ 1',\ 2'\}$ then $\mathbf{2\!<}\!\sqcup\! \mathbf{3\!<}$ is the  poset with elements $\{0,\ 1,\ 0',\ 1',\ 2'\}$ and Hasse diagram as follows:
$$
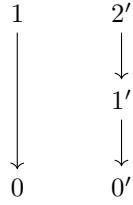
\begin{tikzcd}
1 \arrow[dd] & 2' \arrow[d]   \\
 & 1' \arrow[d]             \\
0           & 0'           
\end{tikzcd}
$$

\indent The main use of disjoint union is breaking down ordered sets into irreducible parts,\ as shown below:

\begin{prop}\label{prop:making prosets into disjoint parts}
Let $\Lambda$ be a non-empty preordered set. Then there is a collection $\{\Lambda_i\}_{i\in I}$ of irreducible subsets of $\Lambda$ such that $\Lambda=\bigsqcup_{i\in I} \Lambda_i$.
\end{prop}

\begin{proof}
Notice that for every $s\in \Lambda$, every element $t$ on the set $\mathcal{N}_{\omega}(s)$ is independent from every element $t'$ on the set $\Lambda\backslash \mathcal{N}_{\omega}(s)$. It is also the case that for every $s\in\Lambda$, the set $\mathcal{N}_{\omega}(s)$ is irreducible.\\
\indent Now let $I\subset \Lambda$ a maximal subset of pairwise independent elements of $\Lambda$ (such set exists by Zorn's Lemma). Define for each $i\in I$ the irreducible subset $\Lambda_i:=\mathcal{N}_{\omega}(i)$. As the set $I$ is maximal, the observation above implies 
$$\Lambda=\bigsqcup_{i\in I}\Lambda_i.$$
\end{proof}

We call the irreducible subsets $\Lambda_i=\mathcal{N}_{\omega}(i)$ the \textbf{components} of $\Lambda$.

\begin{defi}
Let $f:\Lambda_1\rightarrow \Lambda_2$ a function between preordered sets. We say $f$ is an \textbf{bijective order preserving map} if it is a set bijection,\ and for every $s_1,\ s_2\in \Lambda_1$ we have that $s_1\preceq s_2$ if,\ and only if,\ $f(s_1)\preceq f(s_2)$.
\end{defi}

\begin{defi}\label{irreducible partial suborder}
Let $\Lambda$ be a set and $  \preceq_1$,\ $  \preceq_2$ two preorders on the set $\Lambda$. We say $  \preceq_1$ is a \textbf{suborder} of $  \preceq_2$ if $s_1  \preceq_1 s_2$ implies $s_1   \preceq_2 s_2$. If $  \preceq_1$ is a partial order we say it is a \textbf{partial suborder} of $  \preceq_2$. If $  \preceq_1$ gives an irreducible partial order we say it is an \textbf{irreducible partial suborder} of $  \preceq_2$.\\
    \indent Given a collection $\{  \preceq_i\}_{i\in I}$ of preorders of $\Lambda$,\ we define their \textbf{transitive closure} as the minimal preorder $  \preceq_{\cup}:=\bigcup_{i\in I}  \preceq_i$ such that if $s_1  \preceq_i s_2$ for some $i\in I$ then $s_1  \preceq_\cup s_2$.
\end{defi}

The following are useful to describe a proset $\Lambda$ in relation to a proset $\Lambda'$,\ with same elements as $\Lambda$,\ and a collection of finite disjoint subsets. These definitions and results are essential for building the simple topological groups in Subsection \ref{subsection:simple groups}.

\begin{defi}
Given proset $\Lambda$ and,\ for some indexing set $I$,\ $\{S_i\}_{i\in I}$ a collection of disjoint subsets of $\Lambda$,\ we define $(\Lambda+\sum_{i\in I} S_i,\  \preceq_{+})$ the proset with order given by:
$$s_1  \preceq_{+} s_2 \textit{ if }
\left\{
	\begin{array}{l}
		s_1  \preceq s_2 \text{ in }\Lambda \text{,\ or}\\
		\exists i_1,\ i_2, \ldots, i_n\in I,\ \exists t_{i_k},\ q_{i_k} \in S_{i_k} \textit{ such that for }1\leqslant k \leqslant n,\ s_1  \preceq t_{i_1},\\ t_{i_k}\preceq q_{i_k} \textit{ and } q_{i_n}  \preceq s_2.
 	\end{array}
\right.$$
\end{defi}
Notice that under such an order, for each $i\in I$ all elements in $S_i$ are equivalent. This is the minimal order containing $\preceq$ as a suborder and with such property.

\begin{ex}\label{ex:omega complexity proset}
Let $\Lambda=\mathbf{Zig}$ as in Example \ref{ex:omega complexity degree} and let $S_1=\{a_0,\ a_3\}$,\ $S_2=\{a_2,\ a_4\}$. Then the proset $(\Lambda+(S_1+S_2),\  \preceq_+)$ has the following order:
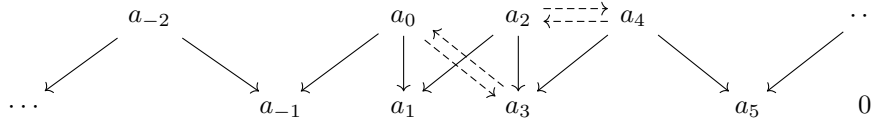
\begin{center}
\begin{tikzcd}
       & a_{-2} \arrow[rd] \arrow[ld] &        & a_0 \arrow[ld] \arrow[rd, dashed, shift right] \arrow[d] & a_2 \arrow[d] \arrow[ld] \arrow[r, dashed, shift left] & a_4 \arrow[ld] \arrow[rd] \arrow[l, dashed, shift left] &     & \cdots \arrow[ld] \\
\cdots &                              & a_{-1} & a_1                                                      & a_3 \arrow[lu, dashed, shift right]                    &                                                         & a_5 &         0         
\end{tikzcd}
\end{center}
where the dashed arrows are the new relations.
\end{ex}

\begin{prop}\label{prop:every proset has poset}
Let $(\Lambda,\  \preceq)$ be a proset. There are $\{S_i\}_{i\in I}$ disjoint subsets of $\Lambda$ and $\leqslant$ a partial suborder of $  \preceq$ such that $(\Lambda,\  \preceq)\cong (\Lambda+\sum_{i\in I}S_i,\ \leqslant_+)$.
\end{prop}

\begin{proof}
First we need to get a partial order $\leqslant$ on $\Lambda$ in such a way we can add equivalence relations and get our proset structure. Let $s_1,\ s_2\in \Lambda$. If $s_1=s_2$ or $s_1 \precnsim s_2$,\ define $s_1\leqslant s_2$. For each $s\in \Lambda$,\ define $\leqslant$ restricted to $\mathcal{N}_0(s)$ as a total order,\ that is,\ for every $s_1,\ s_2\in \mathcal{N}_0(s)$ either $s_1\leqslant s_2$ or $s_2\leqslant s_1$. By definition of $(\Lambda,\ \leqslant)$,\ it is then the case that $(\Lambda,\ \preceq)\cong (\Lambda+\sum_{s\in \Lambda}\mathcal{N}_0(s),\ \leqslant_+)$
\end{proof}

\begin{coro}\label{prop:adding finite sets hence finite interval}
If $\Lambda$ is a locally proset,\ then for every finite subset $S'\subset \Lambda$ the proset $(\Lambda+S',\  \preceq_+)$ is locally finite. 
\end{coro}

\subsection{The units group of a 
matrix ring}

Given $P$ a ring,\ we denote the \textbf{group of units of $P$} as $P^*$. We will denote the ring of $n\times n$ matrices over $P$ by $\textup{M}_n(P)$,\ the group of $n\times n$ invertible matrices by $\textup{GL}_n(P)$,\ and the group of $n\times n$ invertible matrices with determinant $1$ as $\textup{SL}_n(P)$.

\begin{theorem}[Theorem 2.6 \cite{category_theory_notes}]
The functor that maps a ring $P$ to its group of units is right adjoint. 
\end{theorem}

\begin{defi}[Definition 2.1 \cite{tarizadeh2023topological}]\label{thrm:giving topology to unit group}
An \textbf{absolute topological ring} is a topological ring such that its group of units with the subspace topology is a topological groups.
\end{defi}

As topological fields are defined so the inverse map is continuous,\ every topological field is an absolute topological ring.

\begin{lemma}[§§103–105 \cite{dickson1901linear}]\label{Lemma:how to make the simple matrices}
Let $F$ be a field and let $n\in \mathbb{N}$ such that $n\geqslant 2$ ; in the case $n = 2$,\ assume $|F| > 3$. Then every proper normal subgroup $\textup{SL}_n(F)$ is central,\ and every noncentral normal subgroup of $\textup{GL}_n(F)$ contains the groups $\textup{SL}_n(F)$.
\end{lemma}

\begin{theorem}[Theorem 4.19 \cite{matrix-results-adj}]\label{thrm:conditions invertible}
Let $P$ be a commutative ring. Then for every $A\in \textup{M}_n(P)$ we have
$$A \adj(A)=\adj(A)A=\det(A)I_n.$$
Hence,\ $A\in \textup{M}_n(P)$ is invertible if,\ and only if,\ $\det(A)$ is invertible in $P$.
\end{theorem}

\begin{theorem}[Theorem 4.21 \cite{matrix-results-adj}]\label{thrm:invertible as product}
Let $P$ be a commutative ring. If $A\in \textup{GL}_n(R)$ then
$$A^{-1}=\frac{1}{\det(A)}\adj(A).$$
\end{theorem}

\begin{theorem}[Theorem 4.27 \cite{matrix-results-adj}][Cayley-Hamilton Theorem]\label{thrm:cayley-hamilton}
Let $P$ be a commutative ring. If $A\in \textup{GL}_n(R)$ then
$$\adj(A)=-\sum_{i=1}^{n}c_iA^{i-1}$$
where $c_i$ are the coefficients of the characteristic polynomial of $A$.
\end{theorem}

\subsection{Topological groups}

\begin{defi}
Let $G$ be a totally disconnected,\ locally compact,\ second countable group (t.d.l.c.s.c.). We say $G$ is \textbf{regionally elliptic} if every compact subset of $G$ is contained in a compact subgroup.
\end{defi}

\begin{defi}[Introduction \cite{Burger2000}]
Let $G$ be a topological group. We define the \textbf{quasicentre} of $G$ as:
$$\textup{QZ}(G)=\{g\in G:\ C_G(g) \text{ is open in }G\}.$$
\end{defi}

When studying totally disconnected,\ locally compact groups,\ a class of groups that play an important role are the elementary groups. 

\begin{defi}[Definition 1.1\cite{elementary_first}]\label{defi:elementary groups}
The class of \textbf{elementary groups} is the smallest class $\mathcal{E}$ of t.d.l.c.s.c. groups such that:
\begin{enumerate}
    \item[(i)] $\mathcal{E}$ contains all second countable profinite groups and countable discrete groups;
    \item[(ii)] $\mathcal{E}$ is closed under taking group extensions,\ that is,\ if there is $N\trianglelefteq H$ closed subgroup such that $N\in\mathcal{E}$ and $H/N\in\mathcal{E}$ then $H\in\mathcal{E}$;
    \item[(iii)] $\mathcal{E}$ is closed under taking closed subgroups;
    \item[(iv)] $\mathcal{E}$ is closed under taking quotients by closed normal subgroups;
    \item[(v)] If $G$ is a t.d.l.c.s.c. group and $\bigcup_{i\in\mathbb{N}}O_i=G$,\ where $\{O_i\}_{i\in\mathbb{N}}$ is an $\subset$-increasing sequence of open subgroups of $G$ with $O_i\in\mathcal{E}$ for each $i$,\ then $G\in\mathcal{E}$. We say that $\mathcal{E}$ is closed under countable increasing union.
\end{enumerate}
\end{defi}

The class $\mathcal{E}$ of elementary groups admits a well-behaved rank $\xi$ \cite[Lemma 4.12]{elementary_first}. This rank roughly measures how many steps it is necessary to build an elementary group $G$ from the discrete and profinite groups under operations $(ii)$ until $(v)$.

\begin{prop}[Proposition 6.1 \cite{elementary_first}]\label{prop:rank quasicentre}
Let $G$ be a t.d.l.c.s.c. group. Assume $G$ contains an open compact subgroup $U\leqslant G$ such that $\textup{QZ}(U)$ is dense in $U$. Then $G$ is elementary and $\xi(G)\leqslant 3$.
\end{prop}

\begin{defi}
The \textbf{derived series} of a topological group $G$ is the sequence of closed normal groups $\{G^{(n)}\}_{n\in\mathbb{N}}$ defined by $G^{(0)}=G$ and $G^{(n+1)}=\overline{[G^{(n)},G^{(n)}]}$. A group $G$ is \textbf{solvable} if this sequence stabilizes at $\{1\}$ after finitely many steps. A group $G$ is \textbf{residually solvable} if $\bigcap_{n\in \mathbb{N}}G^{(n)}=\{1\}$.
\end{defi}

\begin{theorem}[Theorem 8.1, Proposition 4.19 \cite{elementary_first}]\label{thrm:locally solvable is elementary}
Let $G$ be a t.d.l.c.s.c. group. If there is $U$ an open, compact subgroup of $G$ such that $U$ is solvable, then $G$ is elementary with $\xi(G)<\omega$. 
\end{theorem}



\section{Group of invertible matrices}\label{groups}\label{5.5}

\subsection{\texorpdfstring{$\textup{GL}_\Lambda(P)$}{TEXT} and its properties}

In this section we will work with the group of incidence algebras, and show how to use them to build simple topological groups. In this section we assume $P$ is a commutative ring.

\begin{defi}[Definition 1.1 \cite{Matrices1973}]\label{defi:proset matrices}\index{$\textup{M}_{\Lambda}(P)$|ndx}\index{ring of matrices indexed by a proset|ndx}
Let $P$ be a ring and $(\Lambda,  \preceq)$ a locally finite proset. We define the \textbf{incidence ring of $\Lambda$ over $P$} as
$$\textup{M}_\Lambda(P):=\left\{ (a_{s_1,s_2})\in P^{\Lambda\times \Lambda}
: \textit{ if }a_{s_1,s_2}\neq 0\textit{ then } s_1  \preceq s_2 \right\}$$
with  coordinatewise sum and multiplication defined as follows: given two element $(a_{s_1,s_2}),(b_{s_1,s_2})\in \textup{M}_\Lambda(P)$ then $(a_{s_1,s_2}).(b_{s_1,s_2})=(c_{s_1,s_2})$, where 
$$c_{s_1,s_2}=\sum_{t\in[s_1,s_2]}a_{s_1,t}b_{t,s_2}.$$ 
\end{defi}

 We will denote the elements of $\textup{M}_\Lambda(P)$ as $A,B,C$. We denote $A_{s_1,s_2}$ as the coordinate $(s_1, s_2)$ in the matrix $A$. When talking about the coordinate of the product $A_1A_2\ldots A_n$ we denote it as $(A_1A_2\ldots A_n)_{s_1,s_2}$. Following are some examples of incidence rings.

\begin{ex}
\begin{enumerate}
    \item Given $n\in\mathbb{N}$, the ring $\textup{M}_{\mathbf{n}}(P)$ is the ring of $n\times n$ matrices over the ring $P$ with the usual multiplication.
    \item Given $n\in\mathbb{N}$, the ring $\textup{M}_{\mathbf{n<}}(P)$ is the ring of $n\times n$ upper triangular matrices over $P$ with the usual multiplication.
    \item Given $S$ a set, the ring $\textup{M}_S(P)$ is isomorphic to $\prod_{s\in S} P$ with the coordinate-wise multiplication. 
    \item The ring $\textup{M}_{\mathbf{N}^*_d}(P)$ has elements of the form:

$$
A:=\left(\begin{array}{cccccc}
A_{1,1} & A_{1,2} & A_{1,3} & A_{1,4} & A_{1,5}  & \cdots \\
0 & A_{2,2} & 0 & A_{2,4} & 0 & \cdots\\
0 & 0 & A_{3,3} & 0 & 0  & \cdots \\
0 & 0 & 0 & A_{4,4} & 0 &  \cdots \\
0 & 0 & 0 & 0 & A_{5,5} &  \cdots  \\
0 & 0 & 0 & 0 & 0 &  \cdots \\
\vdots & \vdots  & \vdots & \vdots & \vdots & \ddots
\end{array}\right)
$$

\noindent and the multiplication of two matrices $A,B\in \textup{M}_{\mathbf{N}^*_d}(P)$ is:

$$
A
B:=\left(\begin{array}{ccccccc}
A_{1,1}B_{1,1} & \sum_{d|2}A_{1,d}B_{d,2} & \sum_{d|3}A_{1,d}B_{d,3} & \sum_{d|4}A_{1,d}B_{d,4} & \sum_{d|5}A_{1,d}B_{d,5}  & \cdots \\
0 & A_{2,2}B_{2,2} & 0 & \sum_{d|4}A_{2,d}B_{d,4} & 0 & \cdots\\
0 & 0 & A_{3,3}B_{3,3} & 0 & 0 &  \cdots \\
0 & 0 & 0 & A_{4,4}B_{4,4} & 0 &  \cdots \\
0 & 0 & 0 & 0 & A_{5,5}B_{5,5} &  \cdots  \\
0 & 0 & 0 & 0 & 0 &  \cdots \\
\vdots & \vdots  & \vdots & \vdots & \vdots & \ddots
\end{array}\right)
$$
\end{enumerate}
\end{ex}

\begin{defi}\index{$\textup{GL}_\Lambda(P)$|ndx}\index{general linear group indexed by a proset|ndx}
Given $P$ a ring and $\Lambda$ a proset. We define the \textbf{general linear group of $P$ in relation to $\Lambda$} as:
$$\textup{GL}_\Lambda(P):=\left\{ A\in \textup{M}_\Lambda(P)
:\  A^{-1}\in \textup{M}_\Lambda(P)\right\}.$$
That is,\ the group of units of $\textup{M}_\Lambda(P)$.
\end{defi}

\begin{notation}\label{nota:normal subgroups}\label{defi:normal subgroups}
Let $\Lambda$ a locally finite proset and $P$ a commutative ring. We denote the following normal subgroups:
\begin{itemize}
    \item Given $s_1  \preceq s_2$ in $\Lambda$ we define 
    $$ N^{\Lambda}_{[s_1,s_2]}(P)=\{A\in \textup{GL}_{\Lambda}(P):\  \text{if }s\in [s_1,\ s_2],\ A_{s,s}=1 \text{,\ and if } t_1\neq t_2\in[s_1,\ s_2],\ \ A_{t_1,t_2}= 0 \}.$$ 
    This normal subgroup has quotient 
    $$\textup{GL}_{\Lambda}(P)/N^{\Lambda}_{[s_1,\ s_2]}(P)\cong \textup{GL}_{[s_1,\ s_2]}(P).$$
    \item For a convex set $\Lambda'\subset \Lambda$,\ the normal subgroup
    $$N^{\Lambda}_{\Lambda'}=\bigcap_{s_1  \preceq s_2\in \Lambda'}N^{\Lambda}_{[s_1,\ s_2]}.$$
    This normal subgroup has quotient $$\textup{GL}_{\Lambda}(P)/N^{(\Lambda,\  \preceq)}_{\Lambda'}\cong \textup{GL}_{\Lambda'}(P).$$
    \item For $\{\Lambda_i\}_{i\in I}$ a locally convex set of $\Lambda$ we define the normal subgroup 
    $$N^{\Lambda}_{\{\Lambda_i\}_{i\in I}}=\bigcap_{i\in I}N^{\Lambda}_{\Lambda_i}.$$
    This normal subgroup has quotient 
    $$\textup{GL}_{\Lambda}(P)/N^{\Lambda}_{\{\Lambda_i\}_{i\in I}}\cong \prod_{i\in I} \textup{GL}_{\Lambda_i}(P).$$
\end{itemize} 
\end{notation}

\begin{ex}
Let $\Lambda=\mathbf{3}<$,\ $P=\mathbb{Z}$ and $\alpha=\{0,1\}$. Then $N^{\mathbf{3}<}_\alpha$ is the normal subgroup with elements of the form:
$$\left(\begin{array}{lll}
1 & 0 & a_{0,2} \\
0 & 1 & a_{1,2} \\
0 & 0 & a_{2,2} 
\end{array}\right)$$
where $a_{2,2}\in \{-1,1\}$. 
\end{ex}

\begin{prop}\label{prop:conditions to have inverse on lambda}\cite[Theorem 1.16]{Matrices1973}\index{absolute topological ring|ndx}
Let $P$ be a commutative ring and $\Lambda$ a preordered set,\ and $A\in \textup{M}_\Lambda(P)$. Then $A\in \textup{GL}_\Lambda(P)$ if,\ and only if,\ $AN_{[s,s]}$ is invertible for all $s$ in $\Lambda$. In other words, a matrix $A$ is invertible if for every interval, the finite submatrix with coordinates restricted to the interval is also invertible.
\end{prop}
\begin{coro}\label{coro:matrices absolute topological}
Let $P$ be a commutative, topological ring. Then for every finite preordered set $\Lambda$ we have that $\textup{M}_\Lambda(P)$ is an absolute topological ring if, and only if, $P$ is an absolute topological ring.
\end{coro}
\begin{proof}
If $P$ is an absolute topological ring, by Theorem \ref{thrm:cayley-hamilton} we have that the map $A\mapsto \adj(A)$ is a polynomial, hence it is continuous. It is also the case that $\textup{det}:\textup{M}_\Lambda(P)\rightarrow P$ is a polynomial function on $P^{\Lambda\times \Lambda}$, therefore it is also continuous. As $P$ is an absolute topological ring, if $A\in \textup{GL}_\Lambda(P)$ then the map $\det(A)\mapsto \frac{1}{\det(A)}$ is also continuous. It then follows by Theorem \ref{thrm:invertible as product} that the the map $^{-1}:\textup{GL}_\Lambda(P)\rightarrow \textup{GL}_\Lambda(P)$ given by $A\mapsto A^{-1}$ is continuous.\\
\indent On the other hand, let $\Lambda=\{0\}$ be the ordered set with only one element. As $\textup{GL}_\Lambda(P)\cong P^*$ it follows that if $P$ is not an absolute topological ring then $\textup{M}_\Lambda(P)$ is not an absolute topological ring.
\end{proof}
As $P$ is commutative,\ when $\Lambda$ is a poset we can define $\textup{GL}_\Lambda(P)$ as all the infinite matrices such that the diagonal elements belongs to $P^*$. In the case $P$ is not commutative this is not always the case,\ as shown in articles \cite{inverse_upper_that_is_lower} and \cite{inverse_upper_that_is_lower_2}.

\begin{remark}\label{remark:top_matrix_grps}
    The following result is a correction and a generalization for the statement in \cite[Lemma 4.1]{simple_matrices}, as the proof assumes that if $P$ is a topological ring and $\Lambda$ is a finite preordered set then the group $\textup{GL}_\Lambda(P)$ is a topological group, which is not always the case as seen in Corollary \ref{coro:matrices absolute topological}.
\end{remark}

\begin{prop}
    Let $P$ be an absolute topological ring and $\Lambda$ a locally finite proset. Then the group $\textup{GL}_\Lambda(P)$ is a closed subset of $\prod_{(a,b)\in \Lambda\times \Lambda}P$, where the topology of $\prod_{(a,b)\in \Lambda\times \Lambda}P$ is the product topology. It is also the case that this topology makes $\textup{GL}_\Lambda(P)$ a topological group. 
\end{prop}

\begin{proof}
    The proof is the same as in \cite[Lemma 4.1]{simple_matrices}.
\end{proof}

\begin{remark}\label{remark:nets and open}
    Let $\{X_i\}_{i\in I}$ be a collection of topological spaces. Define $X=\prod_{i\in I}X_i$ a topological space given by the product topology on the sets $X_i$. For each $i\in I$, let $\pi_i:X\rightarrow X_i$ be the projection on the $i$-th coordinate. Then:
    \begin{itemize}
        \item The subsets $O\subset X$ such that for all $i\in I$ the set $\pi_i(O)$ is an open subset of $X_i$, and there is a finite subset $J\subset I$ such that if $i\in I\backslash J$ then $\pi_i(O)=X_i$ form a basis of open subsets of $\prod_{i\in I}X_i$.
        \item Given $\{x_a\}_{a\in A}$ a net in $X$ converges, it converges to $x$ if, and only if, for every $i\in I$ the net $\{\pi_i(x_a)\}_{a\in A}$ converges to $\pi_i(x)$ in $X_i$. 
    \end{itemize}
\end{remark}

The following result follows from the fact that the topology on $\textup{GL}_{\Lambda}(P)$ is given by the product topology.

\begin{coro}\label{coro:open iff matrices group}
Let $P$ a commutative,\ absolute topological ring,\ $\Lambda$ a locally finite proset and $\{\Lambda_i\}_{i\in I}\subset \Lambda$ locally convex. Then the following are true:
\begin{enumerate}
    \item The normal subgroup $N^{\Lambda}_{\{\Lambda_i\}_{i\in I}}$ is closed in $\textup{GL}_\Lambda(P)$.
    \item The normal subgroup $N^{\Lambda}_{\{\Lambda_i\}_{i\in I}}$ is open if,\ and only if,\ $\bigsqcup_{i\in I} \Lambda_i$ is a finite set and $P$ has the discrete topology.
\end{enumerate}
\end{coro}

\subsection{The center}

For the next result,\ given $A\in \textup{M}_\Lambda(P)$,\ we say that $A$ is \textbf{scalar}\index{scalar matrix|ndx} if $A_{s,s}=A_{t,t}$ for every $t,s\in \Lambda$,\ and if $t\neq s$ then $A_{t,s}=0$. We will denote $Z(\textup{M}_\Lambda(P))$\index{center of a ring|ndx}\index{center of a group|ndx} as the center of the ring of matrices and $Z(\textup{GL}_\Lambda(P))$ the center of the group of units of matrices.

\begin{prop}\cite[Theorem 1.23]{Matrices1973}
Let $\Lambda$ be a finite irreducible proset and $P$ a commutative ring. Then $A\in Z(\textup{M}_\Lambda(P))$ if,\ and only if,\ $A$ is scalar.
\end{prop}

\begin{prop}\label{prop:first result on center}
Let $\Lambda$ be a finite irreducible proset and $P$ a commutative ring such that $\exists p_1,\ p_2\in P^*$ with $p_1-p_2\in P^*$. Then $A\in Z(\textup{GL}_\Lambda(P))$ if,\ and only if,\ $A$ is scalar.
\end{prop}

\begin{proof}
First we will use induction to prove for the case $\Lambda$ is a poset. We then use it to generalize for the prosets case.\\
\indent Given that $|\Lambda|=1$,\ both results follow trivially,\ as every element is scalar. If $|\Lambda|=2$,\ as $\Lambda$ is irreducible then $\Lambda=\mathbf{2}<$. Let $A\in \textup{M}_\Lambda(P)$. Then $A$ can be written as follows.
$$A:=\begin{pmatrix}
A_{0,0} & A_{0,1} \\
0 & A_{1,1}
\end{pmatrix}.$$
Assume now that $A\in Z(\textup{GL}_\Lambda(P))$. Hence,\ given $B\in \textup{GL}_{\Lambda}(P)$ any element,\ the following equations are true for every $B_{0,0},B_{0,1},B_{1,1}\in P$
$$ \left\{ \begin{array}{l}
A_{0,0}B_{0,0}=B_{0,0}A_{0,0}\\
A_{1,1}B_{1,1}=B_{1,1}A_{1,1}\\
B_{0,0}A_{0,1}+B_{0,1}A_{1,1}=A_{0,0}B_{0,1}+A_{0,1}B_{1,1}.
\end{array} \right. $$
\indent By making $B_{0,0}=B_{1,1}=B_{0,1}=1$ we get $A_{0,0}=A_{1,1}$. On the other hand,\ by fixing $B_{0,0}=B_{1,1}=1$ and letting $B_{0,1}\in P$ any element,\ it follows that $A_{0,0},\ A_{1,1}\in P^*$. The third equation also tell us that for every $p_1,\ p_2\in P^*$ we have that $A_{0,1}p_1=A_{0,1}p_2$. Hence $A_{0,1}(p_1-p_2)=0$. Using $p_1,\ p_2$ from the assumption and using the fact that units cannot be zero divisors,\ it follows that $A_{0,1}=0$.\\
\indent Assume that for every $k\leqslant n$ it is true that if $|\Lambda|=k$ and $\Lambda$ is an irreducible poset than the statements are true.\\
\indent Let $|\Lambda|=n+1$. Let $s_1\leqslant s_2$. If $|[s_1,\ s_2]|\leq n$,\ then the result follows on the submatrix with coordinates on $[s_1,\ s_2]$ by induction. Hence the result is always true when $|\Lambda|\leqslant n+1$ and $\Lambda\neq \mathbf{n+1}<$.\\
\indent If $\Lambda:=\mathbf{n+1}<$,\ by looking at the submatrices indexed by $\{0,1,\ldots,n-1\}$ and $\{1,2,\ldots,n\}$ it only remains to show $A_{0,\ n}=0$. Given $A\in \textup{M}_{\Lambda}(P)$ then,\ if $AB=BA$ the coordinate $A_{0,\ n}$ implies that
$$\sum_{0\leq i\leq n} A_{0,i}B_{i, n} = \sum_{0\leq i\leq n}B_{0,i}A_{i, n}$$
\indent By the induction argument,\ as $A\in Z(\textup{M}_{\Lambda}(P))$ we have that $A_{0,0}=A_{n,n}\in Z(P)$ and for $i\neq j$ and $(i,\ j)\neq(0,\ n)$ we have that $A_{i,j}=0$. Hence the equation is reduced to
\begin{equation}\label{eq:same thing}
    A_{0,0}B_{0,\ n}+A_{0,n}B_{n,n}=B_{0,0}A_{0, n}+B_{0, n}A_{n,n}
\end{equation}
By letting $B_{0,0}=B_{n,n}=B_{0,n}=1$ it follows that $A_{0,0}=A_{n,n}$. On the other hand,\ by fixing $B_{0,0}=B_{1,1}=1$ and letting $B_{0,n}\in P$ any element,\ it follows that $A_{0,0},\ A_{n,n}\in P^*$. The third equation also tell us that for every $p_1,\ p_2\in P^*$ we have that $A_{0,n}p_1=A_{0,n}p_2$. Hence $A_{0,n}(p_1-p_2)=0$. Using $p_1,\ p_2$ from the assumption and using the fact that units cannot be zero divisors,\ it follows that $A_{0,n}=0$.\\
\indent For the proset case,\ denote the proset as $(\Lambda,\  \preceq)$. Notice that every scalar element $A$ with coordinates in $P^*$ belongs to $\textup{GL}_{(\Lambda,\  \preceq)}(P)$. It is also the case that for any irreducible partial suborder $\leqslant$ of $  \preceq$ (Definition \ref{irreducible partial suborder}),\ if $A\in Z(\textup{GL}_{(\Lambda, \leqslant)})$ then $A$ is scalar. Hence,\ $A$ can only belong to the center of $\textup{GL}_{(\Lambda,\  \preceq)}$ if it is scalar. 
\end{proof}

Note that the group part of the result is not true for the case $\Lambda=\mathbf{2}<$ and $\mathbb{F}_2$,\ as $Z(\textup{M}_\mathbf{2}(\mathbb{F}_2))$ has only scalar matrices but  $\textup{GL}_\mathbf{2}(\mathbb{F}_2)\cong \mathbb{F}_2$ is an abelian group. This also gives us an example that given $P$ a non-commutative ring,\ it is not necessarily true that $Z(P)\cap P^*=Z(P^*)$.

\begin{coro}
Let $\Lambda$ be an irreducible proset and $P$ a commutative,\ absolute topological ring. Then $Z(\textup{M}_{\Lambda}(P))\cong P$. If there are $p_1,\ p_2\in P^*$ such that $p_1-p_2\in P^*$ then $Z(\textup{GL}_\Lambda(P))\cong P^*$.
\end{coro}

\subsection{Building simple topological groups}\label{subsection:simple groups}

\begin{defi}\index{almost general linear group|ndx}\index{$\textup{aGL}_\Lambda(P)$|ndx}
Given $\Lambda$ a poset and $P$ a ring,\ define the \textbf{almost general linear group} of $P$ in relation to $\Lambda$ as
$$\textup{aGL}_\Lambda(P):=\left\{A\in {\textup{GL}_{\Lambda+S}} \text{ for some } S\subset\Lambda \text{ finite} \right\}.$$
In other words,\ the set of all matrices that accept finitely many non-zero entries without relation in $\Lambda$.
\end{defi}

\begin{ex}
Let $\Lambda=\mathbf{N}$. Then $A\in \textup{aGL}_\mathbf{N}(P)$ if,\ and only if,\ $A\in \textup{GL}_{\mathbf{N}+\{0,\ 1,\ldots,\ n\}}(P)$ for some $n\in\mathbb{N}$.
\end{ex}

\begin{prop}\label{prop:conditions for atri to be elementary}
Let $\Lambda$ be a poset and $P$ a ring. The group $\textup{aGL}_\Lambda(P)$ is equal to the direct limit $\varinjlim_{S\subset \Lambda\text{finite}}\textup{GL}_{\Lambda+S}(P)$.
\end{prop}

\begin{proof}
For each $S$ finite subsets of $\Lambda$,\ define $j_S:\textup{GL}_{\Lambda}\rightarrow \textup{GL}_{\Lambda+S}$ the embedding such that $A_{s_1,s_2}=(j_S(A))_{s_1,s_2}$. If $S_1\subset S_2$ are subsets of $\Lambda$,\ define $j_{S_1\rightarrow S_2}:\textup{GL}_{\Lambda+S_1}\rightarrow \textup{GL}_{\Lambda+S_2}$ the embedding such that $A_{s_1,s_2}=(j_{S_1\rightarrow S_2}(A))_{s_1,s_2}$. These morphisms are such that if $S_1\subset S_2\subset S_3$ are finite subsets,\ then $j_{S_1\rightarrow S_3}=j_{S_2\rightarrow S_3}\circ j_{S_1\rightarrow S_2}$ and $j_{S_1}=j_{S_2}\circ j_{S_1\rightarrow S_2}$. Hence the direct limit of this system,\ $\varinjlim \textup{GL}_{\Lambda+S}(P)$ exist. We now need to prove $\varinjlim \textup{GL}_{\Lambda+S}(P)=\textup{aGL}_{\Lambda}(P)$.\\
\indent For each $S$ finite subset of $\Lambda$,\ define $i_S:\textup{GL}_{\Lambda+S}\rightarrow \textup{aGL}_\Lambda(P)$ the embedding such that $A_{s_1,s_2}=(i_{S}(A))_{s_1,s_2}$. Notice that these morphisms are such that the outer triangle of the following diagram commutes:

$$
\begin{tikzcd}
                                                                                            &  & \textup{aGL}_\Lambda(P)                                                   &  &                                                                                                                                                \\
                                                                                            &  &                                                                  &  &                                                                                                                                                \\
                                                                                            &  & \varinjlim \textup{GL}_{\Lambda+S}(P) \arrow[uu, "i" description, dashed] &  &                                                                                                                                                \\
\textup{GL}_{\Lambda+S_2}(P) \arrow[rruuu, "i_{S_2}" description] \arrow[rru, "j_{S_2}" description] &  &                                                                  &  & \textup{GL}_{\Lambda+S_1}(P) \arrow[lluuu, "i_{S_1}" description] \arrow[llu, "j_{S_1}" description] \arrow[llll, "j_{S_1\rightarrow S_2}" description]
\end{tikzcd}$$
\noindent Hence,\ there exists an unique $i:\varinjlim \textup{GL}_{\Lambda+S}(P)\rightarrow \textup{aGL}_\Lambda(P)$. By definition of all the morphisms and of $\textup{aGL}_\Lambda(P)$ it follows that $i$ is an isomorphism.
\end{proof}

The following corollary follows directly from the last proposition.

\begin{coro}
Let $\Lambda$ be a poset and $P$ a ring. If $\Lambda$ is countable,\ $\Lambda=\{s_n\}_{n\in\mathbb{N}}$ is an enumeration then $\textup{aGL}_\Lambda(P)=\bigcup_{n\in\mathbb{N}}\textup{GL}_{\Lambda+\{s_1, s_2, \ldots, s_n\}}(P)$. If $P$ is an absolute topological ring,\ $\textup{aGL}_{\Lambda}(P)$ is a topological group under the final topology.
\end{coro}

For the next result,\ we use the following notation:
\begin{itemize}
    \item Let $S$ be a finite set with cardinality $n$ and $F$ a field. We denote $SL_{\overline{S}}(F)\leqslant \textup{GL}_{\overline{S}}(F)$ the square matrices indexed by $S$ with determinant $1$
    \item Given $\Lambda$ an irreducible proset and $S\subset \Lambda$ a finite set,\ define $(\overline{S},\  \preceq_+)\subset \Lambda$ the minimal convex subset containing $S$ with order given by all elements being equivalent.
\end{itemize}

\begin{theorem}\label{theorem:simple matrices}\index{topologically simple group|ndx}
Let $F$ be a topological field and $\Lambda$ an infinite, locally finite proset. Then $Z_\Lambda(F):=Z(\textup{aGL}_{\Lambda}(F))$ is the unique largest proper closed normal subgroup of $\textup{aGL}_\Lambda(F)$. In particular,\ $\textup{aPGL}_{\Lambda}(P):=\textup{aGL}_\Lambda(F)/Z_\Lambda(F)$ is topologically simple.
\end{theorem}

\begin{proof}
The group $Z_\Lambda(F)$ is clearly closed and normal in $\textup{aGL}_\Lambda(F)$. To complete the proof we will show that if $M$ is a normal subgroup of $\textup{aGL}_\Lambda(F)$ is not contained in $Z_\Lambda(F)$,\ then $M=\textup{aGL}_\Lambda(F)$. For the proof we will look at the intersection of $M$ with the subgroups of $\textup{aGL}_{\Lambda}(F)$ of the form $\textup{GL}_{\Lambda+S}(F)$,\ for $S$ a finite subset of $\Lambda$,\ and show all the elements of $\textup{GL}_{\Lambda+S}(F)$ are in $M$.\\
\indent Fix such $M$ and let $A\in M$ which is not scalar. Let $S$ be a finite set containing $\{s,t \in \Lambda:\ A_{s,t}\neq 0 \textit{ and } s\npreceq t\}$. Then $A$ can be seen as an element in $\textup{GL}_{\Lambda+S}(F)$.\\
\indent Now,\ let $N^{\Lambda+S}_{S}(F)$ the normal subgroup of $\textup{GL}_{\Lambda+S}(F)$ as in Notation \ref{defi:normal subgroups}. For each $S\subset S'\subset \Lambda$ finite convex subset of $\Lambda+S$,\ define: 
$$\textup{M}_{S'}:=(\textup{GL}_{\Lambda+S}(F)\cap M)/N^{\Lambda+S}_{S'}.$$
That can be seen as a normal subgroup of $\textup{GL}_{\overline{S'}}(F)$. Notice that as $AN_{S'}^{\Lambda+S}$ is a non-central element in $\textup{GL}_{\overline{S'}}(F)$,\ hence $\textup{M}_{S'}$ is non-central in $\textup{GL}_{\overline{S'}}(F)$. By Lemma \ref{Lemma:how to make the simple matrices} we have that $SL_{\overline{S'}}(F)\leqslant \textup{M}_{S'}$. We now want to show $\textup{M}_{S'}=\textup{GL}_{\overline{S'}}(F)$.\\
\indent Given $A'\in \textup{GL}_{\overline{S'}}(F)$,\ there is $B$ an element of $\textup{GL}_{\Lambda+S'}(F)$ such that for
$$B_{i,j}=\left\{\begin{array}{ll}
A'_{i,j}     &  \mbox{if }i,j\in S',\\
0     & \mbox{if } i,j\notin S'\mbox{ and }i\neq j,\\
\mbox{is an element of }P^* & \mbox{if }i,j\notin S' \mbox{ and }i=j.
\end{array}\right.$$
Moreover,\ by changing the coordinates $(i,i)$ for $i\notin S'$,\ $B$ may be chosen with $BN^{\Lambda+S'}_{S''}$ is an element $SL_{S''}(F)$ for some finite convex subset $S''$ of $\Lambda+S'$ that strictly contains $\overline{S'}$.\\
\indent Let $A$ the fixed non-scalar element in $M$. Let $A'=AN^{\Lambda+S}_{S'}\in \textup{GL}_{\overline{S'}}(F)$. Let $B$ be the element in $SL_{S''}(F)$ generated in relation to $A'$,\ as seen above. Since $N_{S''}\geqslant SL_{S''}(F)$ and $\textup{GL}_{\overline{S'}}(F)$ can be seen as a subgroup of $SL_{S''}(F)$,\ the previous argument and the fact $A=BN^{\Lambda+S}_{S'}$ implies $\textup{M}_{S'}=\textup{GL}_{\overline{S'}}(F)$. Hence, for every $S\subset \Lambda$ finite subset, $\textup{GL}_{\Lambda+S}(F)\leqslant M$. Proposition \ref{prop:conditions for atri to be elementary} then implies $M=\textup{aGL}_{\Lambda}(F)$.\\
\indent In particular,\ any nontrivial closed normal subgroup of $\textup{aGL}_\Lambda(F)/Z_\Lambda(F)$ has preimage equal to $\textup{aGL}_\Lambda(F)$ and $\textup{aGL}_\Lambda(F)/Z_\Lambda(F)$ is topologically simple.
\end{proof}
Notice that if $\Lambda$ is a finite proset,\ then: 
$$\textup{aGL}_\Lambda(G)/Z_\Lambda(F)=\textup{PGL}_{|\Lambda|}(F);$$
the latter group fails to be topologically simple if $|F|>2$,\ because the determinant homomorphism.\\
\indent One can see that this construction can make simple topological groups of large cardinalities over different fields. Under some conditions the group created can have an interesting topology,\ for example,\ if your field is finite with discrete topology,\ the group generated will be a totally disconnected locally compact group. In general,\ if the field has some property $P$ on the topology then this simple group will be locally pro-$P$.\\
\indent For $\Lambda$ a preordered set and $P$ a ring, define the subgroup
$$aGL_\Lambda(P)=\{A\in \textup{aGL}_\Lambda(P):\ \text{there is }\Lambda'\in \Gamma(\Lambda) \text{ such that }A_{s,s}\neq 1\text{ then }s\in \Lambda'$$
$$\text{ and }A_{s_1,s_2}\neq 0 \text{ then }s_1, s_2\in \Lambda' \}\leqslant \textup{aGL}_\Lambda(P)$$
that is, the subgroup of $\textup{aGL}_\Lambda(P)$ accepting only finitely many non-zero coordinates. Notice that this group is normal and, when $\Lambda$ is infinite, it is a dense proper subgroup of $\textup{aGL}_{\Lambda}(P)$, hence not closed. It then follows that even though $\textup{aPSL}_\Lambda(P)$ is a simple topological group, it is not simple in the group theoretic sense.

\subsection{Matrix groups of finite fields and t.d.l.c.s.c. groups}
 
\begin{coro}
Let $F$ a finite field and $\Lambda$ is an infinite,\ irreducible, locally finite proset. Then the group $\textup{aPGL}_{\Lambda}(F)$ is a topologically simple,\ locally elliptic group.
\end{coro}

\begin{proof}
Theorem \ref{theorem:simple matrices} implies these groups are simple and Proposition \ref{prop:conditions for atri to be elementary} implies it is the countable union of profinite,\ second countable groups. Hence,\ the group is locally elliptic. 
\end{proof}

If for every $s\in \Lambda$ we have that $\mathcal{N}_1(s)$ is finite we say that $\Lambda$ has \textbf{finite neighborhoods}.\index{finite neighborhoods|ndx}

\begin{lemma}\index{quasi-discrete|ndx}
Let $\Lambda$ be an infinite irreducible proset and $P$ a finite commutative ring with discrete topology. If $\Lambda$ has finite neighborhoods then $\textup{QZ}(\textup{GL}_\Lambda(P))$ is dense in $\textup{GL}_\Lambda(P)$.
\end{lemma}

\begin{proof}
Assume $\Lambda$ satisfies the conditions of the lemma. Let $S$ be a finite set of elements of $\Lambda$. Let $\alpha$ the minimal convex subset of $\Lambda$ containing $\cup_{s\in S}\mathcal{N}_1(s)$. As $\Lambda$ is irreducible,\ such subset exists. The union is a finite subset of $\Lambda$,\ hence so is $\alpha$. Define 
$$G_{S}=\{A\in \textup{GL}_\Lambda(F):\  \text{if }t\notin S,\ t'\in\Lambda \text{ then }A_{t,t}=1\ \text{and } A_{t,t'}=A_{t',t}=0\}\leqslant \textup{GL}_\Lambda(P)$$ 
Notice that for every $g\in G_{S}$ we have that 
$$N^{\Lambda}_{\alpha}\subset C_{\textup{GL}_\Lambda(P)}(g).$$
Hence,\ under the assumption on $\Lambda$ and $P$,\ $C_{\textup{GL}_\Lambda(P)}(g)$ is open. As this is the case for every $S$ finite subset of $\Lambda$,\ we have that:
$$H=\langle \cup_{S\subset_{\text{finite}} \Lambda} G_S\rangle\leqslant \textup{QZ}(\textup{GL}_\Lambda(P)).$$
It remains to show that $\textup{QZ}(\textup{GL}_\Lambda(P))$ is dense in $\textup{GL}_\Lambda(P)$. Notice that for every $\beta\in \Gamma(\Lambda)$,\ the map $\pi_{\beta}:\textup{GL}_{\Lambda}(P)\rightarrow \textup{GL}_{\beta}(P)$ given by the quotient of $\textup{GL}_\Lambda(P)$ under $N^\Lambda_\beta$ is such that $\pi_\beta(H)=\pi_{\beta}(G_\beta)=\textup{GL}_\beta(P)$. By the inverse limit property, it is then the case $\textup{GL}_\Lambda(P)=\overline{H}$.
\end{proof}

\begin{coro}\index{elementary group|ndx}
Let $\Lambda$ be a proset that has finite neighborhoods and $P$ be a finite commutative ring with discrete topology. If $G$ is a t.d.l.c.s.c. group such that $\textup{GL}_\Lambda(P)\leqslant G$ is an open,\ compact subgroup,\ then $G$ is an elementary group with $\xi(G)\leqslant 3$.
\end{coro}

\begin{proof}
Notice that $\textup{QZ}(\textup{GL}_\Lambda(P))\leqslant \textup{QZ}(G)$. Hence $\textup{GL}_\Lambda(P)\leqslant\overline{\textup{QZ}(G)}$. By Proposition \ref{prop:rank quasicentre} it follows that $G$ is elementary and $\xi(G)\leqslant 3$.
\end{proof}

A preordered set $\Lambda$ is \textbf{n-bounded} if for every $s\in \Lambda$, the set $\mathcal{N}_1(s)$ has at most $n$ elements.

\begin{lemma}\index{solvable group|ndx}
Assume $\Lambda$ is a partially ordered set and $P$ is a commutative, absolute topological ring. Then the group $\textup{GL}_\Lambda(P)$ is 
residually solvable. If $\Lambda$ is $n$-bounded then $\textup{GL}_\Lambda(P)$ is solvable and $\textup{GL}_\Lambda(P)^{(n)}=\{1\}$.
\end{lemma}

\begin{proof}
Let $s_1\leqslant s_2$ be two elements of $\Lambda$. We will prove by induction that for every $A,\ B\in \textup{GL}_\Lambda(P)^{(k)}$, if $[s_1,\ s_2]$ has less that $k+1$ elements, then 
$$[A,\ B]_{s_1,s_2}=\left\{\begin{array}{ll}
1 & \mbox{if }s_1=s_2,\\
0 & \mbox{otherwise.}
\end{array}\right.$$
By Remark \ref{remark:nets and open} it then follows that for every $C\in \overline{[\textup{GL}_{\Lambda}(P)^{(k)},\ \textup{GL}_{\Lambda}(P)^{(k)}]}=\textup{GL}_{\Lambda
}(P)^{(k+1)}$ and $s_1,\ s_2\in \Lambda$ such that $[s_1,\ s_2]$ has at most $k+1$ elements, then:
$$C_{s_1,s_2}=\left\{\begin{array}{ll}
1 & \mbox{if }s_1=s_2,\\
0 & \mbox{otherwise.}
\end{array}\right.$$
\indent For the case $k=1$, let $A,\ B\in \textup{GL}_{\Lambda}(P)$ and $s\in \Lambda$. As $P$ is commutative it follows that 
$$[A,\ B]_{s_1,s_1}=A_{s_1,s_1}B_{s_1,s_1}A^{-1}_{s_1,s_1}B^{-1}_{s_1,s_1}=1,$$
proving the base case.\\
\indent Assume it is true for all $k\leqslant m$.  Let $k=m+1$ and $s_1,\ s_2\in\Lambda$ be such that $[s_1,\ s_2]$ has $m+1$ elements. Let $A,\ B\in \textup{GL}_\Lambda(P)^{(m)}$. The induction hypothesis implies that for every  $s_1<s_3$ we have $A_{s_1,s_3}=B_{s_1,s_3}=A^{-1}_{s_1,s_3}=B^{-1}_{s_1,s_3}=0$, for every $s_3<s_2$ we have $A_{s_3,s_2}=B_{s_3,s_2}=A^{-1}_{s_3,s_2}=B^{-1}_{s_3,s_2}=0$ and $A_{s_1,s_2}=-A^{-1}_{s_1,s_2}$,  $B_{s_1,s_2}=-B^{-1}_{s_1,s_2}$. Hence:
$$[A,B]_{s_1,\ s_2}=A_{s_1,s_2}+B_{s_1,s_2}-A_{s_1,s_2}-B_{s_1,s_2}=0,$$
proving the induction. It then follows that if $A\in \bigcap_{n\in\mathbb{N}}\textup{GL}_{\Lambda}(P)^{(n)}$, then $A=1^{\Lambda}$, that is, $\textup{GL}_{\Lambda}(P)$ is residually solvable.\\
\indent Assume now that the partially ordered set is $n$-bounded. It then follows that for every element $A$ of $\textup{GL}_{\Lambda}(P)^{(n)}$, $A=1^{\Lambda}$. Hence, for this case, $\textup{GL}_{\Lambda}(P)$ is solvable.
\end{proof}

With Theorem \ref{thrm:locally solvable is elementary} we get the following result:

\begin{coro}\index{elementary group|ndx}
Assume $\Lambda$ is an n-bounded partial ordered set and $P$ is a commutative, profinite, second countable ring. If $G$ is a t.d.l.c.s.c. group such that $\textup{GL}_{\Lambda}(P)\leqslant G$ is an open, compact subgroup, then $G$ is an elementary group with $\xi(G)<\omega$.
\end{coro}

\subsection{Relation to the matrix groups defined in \texorpdfstring{\cite{simple_matrices}}{TEXT}}

The next definition and result comes from \cite{simple_matrices}. There Groenhout,\ Willis and Reid use the definition of being $\mathbb{Z}$-like to prove some results on the unit group of these matrices. The following definition and result will translate it to the terms used here.

\begin{defi}\cite[Definition 2.1]{simple_matrices}
Fix a preordered set $(\Lambda,\  \preceq)$. A subset $\Lambda'\subset \Lambda$ is \textbf{strongly convex}\index{strongly convex|ndx} if for all $s\in \Lambda\backslash \Lambda'$,\ either $s  \precnapprox t$ for all $t\in \Lambda'$. or else $t  \precnapprox s$ for all $t\in \Lambda'$. A proset $(\Lambda,\  \preceq)$ is said to be $\mathbb{Z}$-like if every finite subset of $\Lambda$ is contained in a finite strongly convex subset of $\Lambda$.
\end{defi}

\begin{lemma}
Let $(\Lambda,\  \preceq)$ be a irreducible proset. Then $\Lambda$ is $\mathbb{Z}$-like if,\ and only if,\ it is locally finite and for every $s\in \Lambda$ we have $\mathcal{N}_1(s)=\Lambda$.
\end{lemma}

\begin{proof}
Assume $\Lambda$ is $\mathbb{Z}$-like. Given $s\neq t\in \Lambda$ it then follows that or $s  \preceq t$ or $t  \preceq s$,\ hence $\mathcal{N}_1(s)=\Lambda$ for every $s\in \Lambda$. If there was $s  \preceq t$ such that $[s,t]$ is infinite then there would be no strongly convex finite set containing $s$ and $t$,\ hence $\Lambda$ is locally finite.\\
\indent Assume now $\Lambda$ is locally finite and for all $s\in \Lambda$ we have $\mathcal{N}_1(s)=\Lambda$. Let $s_1,\ s_2, \ldots,\ s_n\in \Lambda$. Under our assumption,\ without loss of generality we can assume $s_1  \preceq s_2  \preceq\ldots  \preceq s_n$. Observe that $\{s_1,\ s_2, \ldots,\ s_n\}\subset [s_1,\ s_n]$,\ and this set is finite by assumption. Let $t\in \Lambda\backslash[s_1,\ s_n]$. By the assumption it is also true that $t  \preceq s_1$,\ $s_1  \preceq t  \preceq s_n$ or $s_n  \preceq t$. Because $t\notin [s_1,\ s_n]$ it is clear that $s_1  \preceq t  \preceq s_n$ is not possible and,\ for every $s\in[s_1,\ s_n]$,\ $s\nsim t$. Hence $t  \precnapprox s_1$ or $s_n  \precnapprox  t$ and $[s_1,\ s_n]$ is $\mathbb{Z}$-like.
\end{proof}

\textbf{Acknowledgements:}\ This article is part of my PhD thesis written under the supervision of George Willis, Colin Reid, and Stephan Tornier. I would like to thank them for advising me while writing the article and proofreading it.\\
\indent \textbf{Conflict of interest} The author declares that there are no conflict of interest.

\bibliography{Bibliography.bib}{}
\bibliographystyle{alpha}

\end{document}